\documentclass[a4paper,11pt]{article}

\usepackage[latin1]{inputenc}
\usepackage[T1]{fontenc}
\usepackage{textcomp}
\usepackage[english]{babel}
\usepackage[tbtags]{amsmath}
\usepackage{amssymb}
\usepackage{amsfonts}
\usepackage{amsthm}
\usepackage{latexsym}
\usepackage{array}
\usepackage{bbm}
\usepackage{multicol}
\usepackage{stmaryrd}
\usepackage{url}

\newcommand{\m}[1]{\mathbbm{#1}}

\newcommand{\q}[1]{\mathcal{#1}}

\newcommand{\e}{\varepsilon}

\DeclareMathOperator{\sgn}{\mathrm{sgn}}
\DeclareMathOperator{\Ai}{\mathrm{Ai}}

\newtheorem{thm}{\rm \bf{Theorem}}
\newtheorem*{thm*}{\rm \bf{Theorem}}

\newtheorem{cor}{\rm \bf{Corollary}}

\newtheorem{defi}{\rm \bf{Definition}}
\newtheorem{nb}{\emph{Remark}}

\makeatletter
\def\blfootnote{\xdef\@thefnmark{}\@footnotetext}
\makeatother

\title{
Scaling-sharp dispersive estimates for the Korteweg-de Vries group}
\date{}%\nodate
\author{Raphaël Côte \and Luis Vega}

\begin{document}

\maketitle

\begin{abstract}
%Nous prouvons des inégalités à poids pour le groupe linéaire de KdV, qui sont
%optimales vis-à-vis du changement d'échelle. Ce type d'inégalités suivent
%l'esprit de 
%celles utilisées pour prover la dispersion linéaires des solutions à
%données petites des 
%équations de KdV généralisées
%
%\bigskip
%
We prove weighted estimates on the linear KdV group, which are scaling
sharp. This kind of estimates are in the spirit of that used to prove
small data scattering for the generalized KdV equations.
\end{abstract}

The purpose of this short note is to give a simple proof of two
dispersive
estimates which are heavily used in the proof of small data scattering
for the generalized Korteweg-de Vries equations \cite{HN98}.

The proof of these estimates can be easily extended to other dispersive
equations.

Denote $U(t)$ the linear Korteweg-de Vries group, i.e  $v=U(t)\phi$ is
the
solution to
\[ \left\{ \begin{array}{l}
v_t + v_{xxx} = 0, \\
v(t=0) = \phi, \end{array} \right. \quad \text{i.e.} \quad
\widehat{U(t) \phi} =
e^{it \xi^3} \hat \phi \quad \text{or} \quad (U(t)\phi)(x) =
\frac{1}{t^{1/3}} \int \Ai \left( \frac{x-y}{t^{1/3}}
\right) \phi(y) dy, \]
where $\Ai$ is the Airy function
\[ \Ai(z) = \frac{1}{\pi} \int_0^\infty \cos \left(\frac{\xi^3} 3 + \xi
z
\right) d\xi. \]

\begin{thm}
Let $\phi,\psi \in L^2$, such that $x\phi, x\psi \in L^2$. Then
\begin{align}
\| U(t)\phi \|_{L^\infty}^2 & \le 2\|\Ai\|_{L^\infty}^2 t^{-2/3} \| \phi
  \|_{L^2} \| x \phi \|_{L^2}, \\
  \| U(t)\phi U(-t)\psi_x \|_{L^\infty} & \le C t^{-1} (\| \phi
\|_{L^2} \| x\psi \|_{L^2} + \| \psi \|_{L^2} \| x\phi \|_{L^2}).
\end{align}
Furthermore, the constant $2\|\Ai\|_{L^\infty}^2$ in the first estimate
is
optimal.
\end{thm}

\begin{nb}
These estimates are often used with $\phi$ replaced by $U(-t) \phi$~:
denoting $J(t) = U(t)xU(-t)$, they take the form
\begin{align}
\| \phi \|_{L^\infty}^2 & \le C t^{-2/3} \| \phi \|_{L^2} \| J(t) \phi
\|_{L^2}, \label{estlinfty} \\
\| \phi \psi_x \|_{L^\infty} & \le C t^{-1} (\| \phi
\|_{L^2} \| J(t) \psi \|_{L^2} + \| \psi \|_{L^2} \| J(t) \phi
\|_{L^2}).
\end{align}
\end{nb}

\begin{proof}
Due to a scaling argument (and representation in term of the Airy
function), we are reduced to show that
\[ \| U(1)\phi \|_{L^\infty} \le C \| \phi \|_{L^2} \| x \phi
\|_{L^2}, \]
and similarly for the second inequality. Hence we consider
\[ (U(1) \phi)(x) = \int \Ai(x-y) \phi(y) dy, \]
and we recall that the Airy function satisfies $|\Ai(x)| \le
C(1+|x|)^{-1/4}$ and $|\Ai'(x)| \le C(1+|x|)^{1/4}$. Then
\begin{align*}
|U(1) \phi|^2(x) & =  \iint \Ai(x-y) \phi(y) \Ai(x-z) \bar\phi(z)
\frac{y-z}{y-z} dydz \\
& = \int \Ai(x-y) y \phi(y) \underbrace{\int \frac{\Ai(x-z)}{y-z} \bar
\phi(z) dz}_{\q H_{z \mapsto y}(\Ai(x- z)\phi(z))(y)} dy \\
& \qquad - \int \Ai(x-z)  z\bar\phi(z) \underbrace{\int
\frac{\Ai(x-y)}{y-z} \phi(y) dy}_{-\q H_{y \mapsto z}(\Ai(x-
y)\phi(y))(z)}
dz \\
& = 2\Re \int \Ai(x-y) y \phi(y) \q H_{z \mapsto y}(\Ai(x-
z)\phi(z))(y) dy,
\end{align*}
where $\q H$ denotes the Hilbert transform (and with the slight abuse of
notation $\frac 1 x$ for $\mathrm{vp} \left ( \frac 1 x \right)$).
As $\q H : L^2 \to L^2$ is isometric and hence continuous (with norm 1),
and $\Ai \in L^\infty$, we get
\begin{align}
|U(1) \phi|^2(x) & \le 2 \| \Ai(x-y) y \phi(y) \|_{L^2(dy)} \| \q
H(\Ai(x-
\cdot)\phi)(y) \|_{L^2(dy)} \label{CSineq} \\
& \le 2 \| \Ai \|_{L^\infty}^2 \| y \phi \|_{L^2} \| \phi
\|_{L^2}. \label{Linftyineq}
\end{align}
This is the first inequality. Let us now prove that the constant is
sharp.

First consider the minimizers in the following Cauchy-Schwarz
inequality~:
\begin{equation} \label{modelineq}
\left| \int y\psi(y) \q H(\psi)(y) dy \right| \le \| y \psi(y)
\|_{L^2(dy)} \| \psi \|_{L^2}.
\end{equation}
There is equality if $y\psi(y) = \lambda \q H(\psi)(y)$ for some
$\lambda
\in \m C$. Then a Fourier Transform
shows that $\partial_\xi \hat \psi(\xi) = \lambda \sgn
\xi \hat \psi(\xi)$, hence $\hat \psi(\xi) = C \exp(-\lambda |x|)$, or
equivalenty, one has equality in (\ref{modelineq}) as soon as
\begin{equation} \label{minimizers}
\psi(y) = \frac{C}{1 + (y/\lambda)^2} \quad \text{for some} \quad
\lambda,C
\in \m R.
\end{equation}
(Notice that all the functions involved lie in $L^2$)

We now go back to (\ref{Linftyineq}). Let $x_0 \in \m R$ where $|\Ai|$
reaches its maximum. Now as $\Ai(x_0) \ne 0$, let $\e>0$ such that for
all
$y \in [-\e,\e]$,
$|\Ai(x_0 -y)| \ge |\Ai(x_0)|/2$, and consider the sequence of functions
\[ \phi_n(x) = \frac{\sqrt n}{1+ (ny)^2} \frac{\m 1_{|y| \le
\e}}{\Ai(x_0-y)}. \]
Denote $\psi_n(y) = \frac{\m 1_{|y| \le n\e}}{1 + y^2}$. As $\Ai(x_0 -y)
\phi_n(ny)= \sqrt n \psi_n(ny)$,
\begin{align*}
|U(1)\phi_n|^2(x_0) & = 2 \int y \sqrt n \psi_n(ny) \q H_{z \to
y}(\sqrt n
\psi_n(nz))(y) dy & = \frac 2 n \int y \psi_n(ny) \q
H(\psi_n)(ny) dy.
\end{align*}
One easily
sees that $\psi_n(y) \to \frac{1}{1+|y|^2}$ in $L^2$ and $y\psi_n(y) \to
\frac{y}{1+|y|^2}$ in $L^2$, and hence, in view of (\ref{minimizers}),
as $\q
H$ is homogeneous of degree 0 and $L^2$ isometric, we have
\begin{align*}
|U(1)\phi_n|^2(x_0) & \sim \frac 2 n \int \frac{y}{1+|y|^2} \q
H(\frac{1}{1+ |\cdot|^2})(y) dy \sim \frac 2 n \| \frac{y}{1+|y|^2}
\|_{L^2} \|  \frac{1}{1+ |y|^2} \|_{L^2} \\
& \sim 2 \| y \sqrt n \psi_n(ny) \|_{L^2} \| \sqrt n \psi_n(ny)
\|_{L^2} \\
& \sim 2 \| y \Ai(x_0-y) \phi_n(y) \|_{L^2} \| \Ai(x_0-y) \phi_n(y)
\|_{L^2}.
\end{align*}
As $\phi_n$ concentrates at point $0$, we deduce
\begin{equation} \label{U(1)phi_n}
|U(1)\phi_n|^2(x_0) \sim 2 | \Ai(x_0) |^2 \| y\phi_n(y)
  \|_{L^2} \| \phi_n(y)\|_{L^2} \quad \text{as } n \to \infty,
\end{equation}
which proves that the sharp constant in the first inequality is $2 \|
\Ai
\|_{L^\infty}^2$.

\bigskip

For the second inequality (estimate of the
derivative), we have as for the first inequality~:
\begin{align*}
(U(1)\phi U(1)\bar\psi_x)(x) & = \iint \Ai(x-y) \phi(y) \Ai'(x-z)
\bar\psi(z)
\frac{y-z}{y-z} dydz \\
& = \int \Ai'(x-y) y\phi(y)  \left( \int \frac{\Ai(x-z)\phi(z)}{y-z} dz
\right)  dy \\
& \qquad - \int_z \Ai'(x-y) \phi(y) \left( \int
\frac{\Ai(x-z)z\bar\psi(z)}{z-y} dz \right) dy \\
& = \int  \Ai'(x-y) y\phi(y) \q H_{z \mapsto y}(\Ai(x- z) \phi(z))(y)
dy \\
& \qquad + \int \Ai'(x-y)
\phi(y) \q H_{z \to y}(\Ai(x-z) z\bar \psi(z))(y) dy.
\end{align*}
Denote $\omega_x(y) = \frac{1}{\sqrt{1+|x-y|}}$~: $\omega_x^{-1} \in
A_2$
(with the notation of \cite{Ste70b}),
so that there exists $C$ not depending on $x$ such that
\[ \forall v, \quad \int | \q H v|^2 \omega_x^{-1} \le C \int |v|^2
\omega_x^{-1}. \]
Recall the well-know asymptotic $|\Ai'(x)| \le C(1 + |x|^{1/4})$. Then
\begin{align*}
\lefteqn{ \left| \int \Ai'(x-y) y\phi(y) \q H(\Ai(x- \cdot)
\bar\psi)(y) dy
\right| } \\
& \le  \left( \int |\Ai'(x-y) y\phi(y)|^2
  \omega_x dy \right)^{1/2}  \left(  \int |\q H(\Ai(x- \cdot)
\bar\psi)(y)|^2
  \omega_x^{-1}(y) dy\right)^{1/2} \\
& \le C \| y \phi(x) \|_{L^2} \left(  \int  |\Ai(x- y) \bar \psi(y)|^2
  \omega_y^{-1}(y) dy \right)^{1/2} \\
& \le C \| y \phi \|_{L^2} \| \psi \|_{L^2}.
\end{align*}
In the same way,
\begin{align*}
\lefteqn{ \left| \int \Ai'(x-y) \phi(y) \q H_{z \to y}(\Ai(x-z)
z\bar\phi(z))(y) dy \right| } \\
& \le \left(  \int |\Ai'(x-y) \phi(y)|^2 \omega_x(y)  \right)^{1/2}
\left(  \int  |\q H_{z \to y}(\Ai(x- z) z \psi(z))(y)|^2
  \omega_x^{-1}(y) dy \right)^{1/2} \\
& \le C \| \phi \|_{L^2} \left(  \int  |\Ai(x- y) y \psi(y) |^2
\omega_x^{-1}(y) dy \right)^{1/2} \\
& \le  C \| \phi \|_{L^2}  \| y \psi \|_{L^2}.
\end{align*}
So that :
\[ \| U(1)\phi U(1)\bar\psi_x \|_{L^\infty} \le 2C ( \| \phi \|_{L^2}
\| x \psi
\|_{L^2} + \| \psi \|_{L^2} \| x \phi \|_{L^2}). \]
Up to scaling and replacing $\psi$ by $\bar \psi$, this is the second
inequality.
\end{proof}

\begin{nb}
This proof (especially (\ref{CSineq})) is reminiscent of that in
\cite{LL01} (see also \cite{DV07})
\[ \| \phi \|_{L^\infty}^2 \le \| \phi \|_{L^2} \| \phi' \|_{L^2}, \]
where the constant is sharp and minimizers are $C e^{-\lambda |x|}$.
This has application to the Schrödinger group  $\q U(t)$ (i.e
$\widehat{\q U(t) \phi} = e^{it\xi^2} \hat \phi$). We have the
following Schrödinger version of estimate (\ref{estlinfty}) (notice that
$\q U(t) x \q U(-t) = e^{\frac{i|x|^2}{4t}} \frac{it}{2} \partial_x
e^{-\frac{i|x|^2}{4t}}$) ~:
\[ \| \psi \|_{L^\infty}^2  = \| e^{\frac{i|x|^2}{4t}} \psi
\|_{L^\infty}^2
\le \| e^{i \frac{|x|^2}{4t}} \psi \|_{L^2} \| e^{i \frac{|x|^2}{4t}}
\partial_x e^{i \frac{|x|^2}{4t}} \psi \|_{L^2}
  \le \frac{2}{t} \| \psi \|_{L^2} \| \q U(t) x \q U(-t) \psi \|_{L^2}.
\]
\end{nb}

From Theorem 1, we can easily obtain the optimal decay in a scaling
sharp Besov like space. Let $\varphi \in \q D(\m R)$ be non-negative
with
support in $]-2,2[$ and such that $\varphi$ equals $1$ in a
neighbourhood
of $[-1.5,1.5]$. Denote $\psi(x) = \varphi(2x) -\varphi(x)$ and
$\psi_j(x)
= \psi(x/2^j)$. Finally introduce
\[ \| \phi \|_{N_t} = \sum_{j = -\infty}^\infty 2^{j/2} \| \psi_j U(-t)
\phi \|_{L^2}. \]
\begin{cor}
We have~:
\[ \| \phi \|_{L^\infty} \le C t^{-1/3} \| \phi \|_{N_t}. \]
\end{cor}

\begin{proof}
Notice that $|x \psi_j(x)| \le 2^{j+1} \psi_j(x)$. As $\phi = \sum_j
U(t) \psi_j U(-t)\phi$, we have~:
\begin{align*}
\| \phi \|_{L^\infty} & \le \sum_j \| U(t) \psi_j U(-t)\phi
\|_{L^\infty}  \\
& \le  C t^{-1/3} \sum_j \| U(t) \psi_j U(-t) \phi
\|_{L^2}^{1/2} \| U(t) x U(-t) U(t) \psi_j U(-t) \phi
\|_{L^2}^{1/2} \\
& \le C t^{-1/3} \sum_j  \|  U(t) \psi_j U(-t) \phi
\|_{L^2}^{1/2} \| x \psi_j U(-t) \phi\|_{L^2}^{1/2} \\
& \le C t^{-1/3} \sum_j \|  U(t) \psi_j U(-t) \phi
\|_{L^2}^{1/2} 2^{j/2} \| U(t) \psi_j U(-t) \phi\|_{L^2}^{1/2} \\
& \le C t^{-1/3} \| \phi \|_{N_t}. \qedhere
\end{align*}
\end{proof}

\normalsize

%\addcontentsline{toc}{section}{References}
%\bibliographystyle{amsplain}
%\bibliography{../../references}

\bigskip
\bigskip

\noindent \textsc{Raphaël Côte}\\
Centre de Mathématiques Laurent Schwartz, École polytechnique\\
91128 Palaiseau Cedex, France\\
\url{cote@math.polytechnique.fr}

\bigskip

\noindent \textsc{Luis Vega}\\
Departamento de Matemáticas, Universidad del País Vasco\\
Aptdo. 644, 48080 Bilbao, España\\
\url{luis.vega@ehu.es}

\end{document}